\documentclass[11pt]{amsart}
\usepackage{amsmath,amssymb,amsthm}
\usepackage[all]{xy}
\usepackage{enumerate}
\usepackage{verbatim}

\theoremstyle{plain}
\newtheorem{theorem}{Theorem}
\newtheorem{proposition}[theorem]{Proposition}

\newtheorem{corollary}[theorem]{Corollary}

\theoremstyle{definition}

\newtheorem{remark}[theorem]{Remark}
\newtheorem{example}[theorem]{Example}

\newcommand{\nc}{\newcommand}
\newcommand{\DMO}{\DeclareMathOperator}

\nc{\Nn}{\mathbb{N}} \nc{\Rr}{\mathbb{R}} \nc{\Zz}{\mathbb{Z}}
\nc{\Qq}{\mathbb{Q}}

\nc{\floor}[1]{\lfloor #1 \rfloor} \nc{\ceil}[1]{\lceil #1 \rceil}
\nc{\norm}[1]{\left\| #1 \right\|}
\nc{\ol}{\overline}
\nc{\ds}{\displaystyle}

\nc{\sgn}{\mathrm{sgn}}

\nc{\fprod}{\Omega_1 \otimes_{\Omega_0} \Omega_2}

\nc{\sfprod}{\Omega_1 \times_{\Omega_0} \Omega_2}

\nc{\gv}[1]{\left\langle #1 \right\rangle}
\DMO{\pr}{pr}
\DMO{\E}{E}
\DMO{\cov}{cov}
\DMO{\var}{var}
\DMO{\sd}{SD}
\DMO{\at}{at}
\DMO{\Tr}{Tr}
\DMO{\Hom}{Hom}
\DMO{\ev}{ev}
\DMO{\coev}{coev}
\DMO{\id}{id}

\nc{\cA}{\mathcal{A}}
\nc{\cB}{\mathcal{B}}
\nc{\cP}{\mathcal{P}}
\nc{\atA}{\Omega_{\cA}}
\nc{\atB}{\Omega_{\cB}}
\nc{\atX}{\Omega_{\cX}}

\title[Bundles of Prob.\ Schemes]{Bundles of Probability Schemes}
\author{Wai Yan Pong}
\address{Department of Mathematics\\ California State University Dominguez
Hills, Carson, CA 90747, USA}
\date{June 21, 2026}
\email{wpong@csudh.edu}

\begin{document}

\begin{abstract}
  We study finite probability through a category whose objects are finite
  probability schemes and whose morphisms are probability-preserving maps,
  called \emph{bundles}.  A bundle simultaneously records a quotient of a
  sample space, the corresponding algebra of random variables, and the
  conditional schemes on its fibers.  The pullback, transfer, and fiberwise
  averaging maps associated with a bundle give a functorial construction of
  conditional expectation and its projection properties.  This recovers the
  tower property, total expectation, total covariance and variance, components
  of variance, and a relative weak law of large numbers.  Fiber products of
  bundles encode conditional independence, and iterated fiber products describe
  finite Markov chains from their adjacent-pair distribution schemes.  Finally,
  the transfer map is identified with a relative trace, so the projection
  formula and base-change identity become instances of Frobenius reciprocity
  and Beck--Chevalley condition.
\end{abstract}

\maketitle

Many constructions in probability are naturally performed relative to a chosen
algebra of random variables.  Nelson emphasized this point in his
book~\cite{nelson}, writing that
\begin{quote}
  every construct or theorem of probability theory can be relativized to
  any algebra $\cA$ of random variables.
\end{quote}
He also described the underlying picture as one probability space fibered over
another.  This article develops that picture for finite probability spaces
with strictly positive distributions.  A quotient of such a space is not only a
partition: the quotient set carries an induced distribution, and the quotient
map preserves probabilities.  We call these probability-preserving maps of
finite schemes \emph{bundles}.

The point of keeping the map, rather than only the partition or the algebra of
functions constant on its fibers, is that the standard operations become
functorial.  A bundle
\[
  \pi\colon \Omega\longrightarrow \Omega'
\]
has a pullback $\pi^{\sharp}$, a transfer $\pi_*$, and a fiberwise averaging
map $\pi_{\flat}$.  The operator $\pi^{\sharp}\pi_{\flat}$ is conditional
expectation onto the algebra of functions constant on the fibers, while
$\pi_{\flat}$ is also the observable-side action of the Bayesian inverse
Markov kernel of $\pi$.  The adjunction with $\pi^{\sharp}$ and the
orthogonal-projection property follow from elementary identities for these
maps.  In this way, bundles package together three familiar descriptions of
the same finite data: quotients of a sample space, finite algebras of random
variables, and partitions.

The category also has a useful fiber product.  Over a common base, the fiber
product takes independent products inside each fiber; ordinary independence is
the special case over the terminal scheme.  This construction is not a
categorical product in the category of bundles, but it is symmetric and
associative in each slice.  Probabilistically, the fiber product is exactly
conditional independence: the distribution scheme of $(X,Y,Z)$ maps
isomorphically to $[X,Z]\otimes_{[Z]}[Y,Z]$ precisely when $X$ and $Y$ are
conditionally independent given $Z$.  Iterating the construction then expresses
finite Markov chains as fiber products of their adjacent-pair distribution
schemes, separating local adjacent-pair data from the rule that glues them into
a full chain.  It also immediately explains why reversing a finite Markov chain
gives another Markov chain.

The appendix reinterprets the transfer map as a relative trace.  Under the
assignment $\Omega\mapsto \Rr^\Omega$, pullback is restriction of scalars and
transfer is fiberwise trace.  From this point of view, the projection formula is
Frobenius reciprocity and the base-change identity for fiber products is the
Beck--Chevalley, or Mackey, condition.  Thus the elementary finite-probability
calculations in the main text fit the formal pattern of a finite commutative
Green functor.

This perspective is close in spirit to categorical treatments of conditioning:
for example, the triangle-fill-in formulation of Furber and Jacobs in the
Kleisli category of the distribution monad~\cite{furber-jacobs-condprob} and
the conditional-expectation functor of Adachi and Ryu~\cite{adachi-ryu-prob}.
The point here is narrower and more concrete: finite probability-preserving
maps already organize conditional expectation, orthogonality, fiber products,
conditional independence, Markov chains, and transfer.

\section{Fiber Bundles}

In this article, a (\emph{probability}) \emph{scheme} is a finite set $\Omega$
equipped with a (\emph{probability}) \emph{distribution}, that is, a function
\[
  \pr \colon \Omega \to (0, \infty), \qquad
  \sum_{\omega \in \Omega} \pr(\omega) = 1.
\]
A \emph{random function} is a function defined on a scheme, and a \emph{random
variable} is a real-valued random function.  The set of random variables
on a scheme $\Omega$, denoted $\Rr^{\Omega}$, forms a real algebra under
pointwise addition and multiplication, where constant random variables are
identified with real numbers.  Throughout this article, an $\Rr$-subalgebra
means a unital $\Rr$-subalgebra.

The \emph{expectation}, \emph{expected value}, or \emph{mean} of a random
variable $X$ on $\gv{\Omega, \pr}$ is defined as
\[
  \E_{\pr}X = \E{X} := \sum_{\omega \in \Omega} X(\omega)\pr(\omega).
\]
One readily verifies that $\E$ is a linear functional on
$\Rr^{\Omega}$.  Moreover, since distributions are strictly positive,
the pairing $\gv{X,Y}_{\Omega}:= \E(XY)$ defines an inner product
on $\Rr^{\Omega}$.  Subsets of a scheme are called \emph{events}.  The
\emph{indicator function} of an event $A \subseteq \Omega$ is the random
variable $1_A$ taking the value $1$ on $A$ and $0$ on $\Omega \setminus
A$.  The \emph{probability} of an event $A$, denoted by $\Pr(A)$, is simply
the expected value of its indicator function.

A (\emph{fiber}) \emph{bundle} is a map $\pi$ between the underlying sets of
two schemes that preserves the probability of events; that is, the probability
of the preimage of any event under $\pi$ equals the probability of the event
itself.  Since distributions are strictly positive, any such map must be
surjective.  Schemes with bundles as morphisms form a category under function
composition with the one-point scheme $\bullet$ as the terminal object.  For
a bundle $\pi \colon \gv{\Omega,\pr} \to \gv{\Omega',\pr'}$, we refer to
$\gv{\Omega,\pr}$ as the \emph{total scheme} and $\gv{\Omega',\pr'}$ as the
\emph{base scheme}.  For each $\omega' \in \Omega'$, the \emph{fiber scheme}
of $\pi$ over $\omega'$, denoted by $\pi_{\omega'}$, is the scheme on the
fiber $\pi^{-1}(\omega')$ equipped with the distribution $\pr/\pr'(\omega')$.

\begin{example}
  The \emph{uniform scheme} on a finite non-empty set $T$, denoted by
  $\Omega T$, consists of the set $T$ equipped with the uniform distribution
  $\pr(\omega) = 1/|T|$.  We write $\Omega[n]$ for the uniform scheme on
  $\{1,\ldots, n\}$.\label{ex:uni-sch}
\end{example}

\begin{example}
  The (\emph{distribution}) \emph{scheme} of a random function $X$ on
  $\gv{\Omega,\pr}$ is the scheme $[X] := \gv{X(\Omega), \pr_X}$, where
  $\pr_X(x) := \Pr(X=x)$ is the probability mass function of $X$ with its
  codomain restricted to $X(\Omega)$.  This makes the restriction of $X$ to its
  range $\Omega \to [X]$ a bundle.  As an illustration, consider the uniform
  scheme $\Omega[6]$ and the events $A=\{1,2,3\}$ and $B=\{1,3,5\}$.  The
  schemes $[1_A]$ and $[1_B]$ are both the uniform scheme $\Omega\{0,1\}$
  while $[\gv{1_A, 1_B}]$ is the scheme on $\{0,1\}^2$ with the following
  distribution:
\begin{center}
      \begin{tabular}{c|cccc}
        $\omega$      & $\gv{0,0}$ & $\gv{0,1}$ & $\gv{1,0}$ &
        $\gv{1,1}$ \\ \hline
        $\pr'(\omega)$ & 1/3 & 1/6 & 1/6 & 1/3
   \end{tabular}
 \end{center}
  \label{ex:distr-sch}
\end{example}

\subsection{Three Functors}
\label{sec:three_functors}

A bundle $\pi \colon \gv{\Omega,\pr} \to \gv{\Omega',\pr'}$ identifies
points within a fiber with a single point in the base scheme.  Thus, for
random variables, $\pi$ induces an $\Rr$-algebra map $\pi^{\sharp}(Z):=Z\circ
\pi$ that identifies $\Rr^{\Omega'}$ with the subalgebra of $\Rr^{\Omega}$
consisting of random variables that are constant on the fibers of $\pi$:
 \begin{align*}
   \pi^{\sharp}(\Rr^{\Omega'}) &= \pi^{\sharp}(\Rr[1_{\omega'} \colon \omega'
   \in \Omega']) = \Rr[1_A \colon A\ \text{is a fiber of}\ \pi].
\end{align*}
 We call $\cA_{\pi}:=\pi^{\sharp}(\Rr^{\Omega'})$ the \emph{associated algebra}
 of $\pi$.  For $A,A'$ fibers of $\pi$, note that $1_A^2 = 1_A$ and $1_A1_{A'}
 = 0$ whenever $A \neq A'$.  Consequently, these $1_A$'s form an orthogonal
 basis of their linear span, which is precisely the algebra $\cA_{\pi}$.  In
 particular, the functions $1_{\omega}$ ($\omega \in \Omega$) form a basis
 of $\Rr^{\Omega}$, and hence the assignment
 $1_\omega \mapsto 1_{\pi(\omega)}$ induces a unique linear map
 $\pi_*$, which we call the \emph{transfer map}, from $\Rr^{\Omega}$
 to $\Rr^{\Omega'}$.  Explicitly,
 \[
   (\pi_*X)(\omega')=\sum_{\omega\in\pi^{-1}(\omega')}X(\omega).
 \]
 Regard $\Rr^{\Omega}$ as an $\Rr^{\Omega'}$-module through $\pi^{\sharp}$,
 so that $Z\in\Rr^{\Omega'}$ acts on $X\in\Rr^{\Omega}$ by $(\pi^{\sharp}Z)X$.
 Then a direct computation verifies that $\pi_*$ is not only a $\Rr$-linear
 map but an $\Rr^{\Omega'}$-module map:
 \begin{equation}
   \pi_*((\pi^{\sharp}Z)X)=Z\,\pi_*X.
   \label{eq:proj_formula}
 \end{equation}

For a scheme $\gv{\Omega,\pr}$, let $\pi_{\Omega}^{\bullet}$ denote the
unique bundle $\Omega \to \bullet$.  The expectation on $\gv{\Omega,\pr}$
is the linear map
\[
  \E = \pi_{\Omega\flat}^{\bullet}
    := \pi_{\Omega*}^{\bullet} m_{\pr}
    \colon \Rr^{\Omega}\to \Rr^{\bullet}=\Rr,
\]
obtained by applying $\pi_{\Omega *}^{\bullet}$ after multiplication by
$\pr$.  For a general bundle $\pi \colon \gv{\Omega, \pr} \to \gv{\Omega',
\pr'}$, compatibility with the factorization $\pi_{\Omega}^{\bullet} =
\pi_{\Omega'}^{\bullet} \pi$ dictates the definition of $\pi_{\flat}$:
\begin{align*}
  \pi_{\Omega \flat}^{\bullet} &= \pi_{\Omega *}^{\bullet}m_{\pr}
   = \pi_{\Omega' *}^{\bullet}\pi_{*}m_{\pr}  \\ &=
   \pi_{\Omega'*}^{\bullet}m_{\pr'}\bigl(m_{1/\pr'}\pi_{*}m_{\pr}\bigr)=
  \pi_{\Omega'\flat}^{\bullet}
     \bigl(m_{1/\pr'}\pi_{*}m_{\pr}\bigr).
\end{align*}
Thus, $\pi_{\flat}=m_{1/\pr'}\pi_{*}m_{\pr}$ and so
\[
  \pi_{\flat}X=\frac{\pi_*(X\pr)}{\pr'}=\frac{\pi_*(X\pr)}{\pi_*\pr}.
\]
In other words, $\pi_{\flat}X$ is the random variable on $\Omega'$ whose value
at $\omega'\in\Omega'$ is $\E(X|\pi^{-1}(\omega'))$, the expectation of $X$
on the fiber scheme over $\omega'$.

It is straightforward to verify that $\pi \mapsto \pi^{\sharp}$ is a
contravariant functor while $\pi \mapsto \pi_*$ and $\pi \mapsto \pi_{\flat}$
are covariant functors from the category of bundles to the category of
finite-dimensional real vector spaces.  We summarize their essential properties
in the following theorem.
 \begin{theorem}
  For any bundle $\pi \colon \gv{\Omega,\pr} \to \gv{\Omega',\pr'}$,
  \begin{enumerate}[i.]
    \item $\pi_{\flat}\pi^{\sharp}$ is the identity map on $\Rr^{\Omega'}$.
      \label{i:section}
    \item Both $\pi^{\sharp}$ and $\pi_{\flat}$ preserve expectation.
      \label{i:preserve-E}
    \item $\pi_{\flat}$ is an $\Rr^{\Omega'}$-module map. \label{i:proj_flat}
    \item The operators $\pi^{\sharp}$ and $\pi_{\flat}$ form an adjoint pair.
      \label{i:adjoint}
    \item $\E_{\pi}:=\pi^{\sharp}\pi_{\flat}$ is the orthogonal projection of
      $\Rr^{\Omega}$ to $\cA_{\pi}$. \label{i:orth_proj}
    \item $\pi^{\sharp}$ is an isometry. \label{i:isometry}
  \end{enumerate}
  \label{th:flat-sharp}
\end{theorem}
\begin{proof}
  The identity $\pi_{\flat}\pi^{\sharp}=\id_{\Rr^{\Omega'}}$ follows directly
  from the definitions of $\pi^{\sharp}$ and $\pi_{\flat}$.  Applying the
  $\flat$ functor to the factorization
  $\pi_{\Omega}^{\bullet}=\pi_{\Omega'}^{\bullet}\pi$ yields
  \[
    \E = \pi_{\Omega\flat}^{\bullet} = \pi_{\Omega'\flat}^{\bullet}\pi_{\flat}
    = \E'\pi_{\flat}.
  \]
  Thus $\pi_{\flat}$ preserves expectation; composing with $\pi^{\sharp}$
  and using~\eqref{i:section} shows that $\pi^{\sharp}$ does as well. The
  assertion in~\eqref{i:proj_flat} follows immediately from the $\pi_{\flat}$
  version of the projection formula~\eqref{eq:proj_formula}:
  \begin{equation}
    \pi_{\flat}(\pi^{\sharp}ZX) = \frac{\pi_{*}(\pi^{\sharp} Z X
    \pr)}{\pi_*\pr} = \frac{Z\pi_{*}(X\pr)}{\pi_*\pr} = Z \pi_{\flat}X.
    \label{eq:flat_proj}
  \end{equation}
  Applying $\pi^{\bullet}_{\Omega'\flat}$ to both sides of this identity
  yields~\eqref{i:adjoint}:
  \begin{align*}
    \gv{Z, \pi_{\flat}X}_{\Omega'}
    &=\pi^{\bullet}_{\Omega'\flat}(Z\pi_{\flat}X)
    = \pi^{\bullet}_{\Omega'\flat}\pi_{\flat}(\pi^{\sharp}ZX) \\
    &= \pi^{\bullet}_{\Omega\flat}(\pi^{\sharp}ZX)
    = \langle \pi^{\sharp}Z,X\rangle_{\Omega}.
  \end{align*}
  To prove~\eqref{i:orth_proj}, first note that $\pi_{\flat}$ is surjective
  since by~\eqref{i:section} $\pi^{\sharp}$ is its right inverse. Thus, the
  range of $\E_{\pi} = \pi^{\sharp}\pi_{\flat}$ is $\cA_{\pi}$. For any $X \in
  \Rr^\Omega$ and $Y \in \cA_{\pi}$, $Y=\pi^{\sharp}Z$ for some $Z \in
  \Rr^{\Omega'}$.  By the adjunction just proved,
  \begin{align*}
    \langle X-\E_{\pi}X, Y \rangle_{\Omega}
    &= \langle X-\pi^{\sharp}\pi_{\flat}X, \pi^{\sharp}Z \rangle_{\Omega} \\
    &= \langle \pi_{\flat}X-\pi_{\flat}\pi^{\sharp}\pi_{\flat}X,
    Z \rangle_{\Omega'} \\
    &= \langle \pi_{\flat}X-\pi_{\flat}X,Z \rangle_{\Omega'} = 0,
  \end{align*}
  Therefore $X-\E_{\pi}X$ is orthogonal to $\cA_{\pi}$ for every $X$, and this
  proves $\E_{\pi}$ is the orthogonal projection onto $\cA_{\pi}$.  Finally,
  for any $Z,Z'\in\Rr^{\Omega'}$, since $\pi^{\sharp}$ is an algebra map,
  it follows from~\eqref{i:preserve-E} that
  \begin{align*}
    \langle \pi^{\sharp}Z, \pi^{\sharp}Z'\rangle_{\Omega} &=
    \E\pi^{\sharp}Z\pi^{\sharp}Z' = \E\pi^{\sharp}(ZZ') = \E'(ZZ') =
    \gv{Z,Z'}_{\Omega'}.
  \end{align*}
  This proves~\eqref{i:isometry}.
\end{proof}
\begin{remark}
The preceding theorem has a useful Markov-kernel interpretation.  For background
on Bayesian inversion and Markov kernels in categorical probability, the reader
can consult~\cite{cho-jacobs-disintegration, fritz-markov-kernels}.  The
deterministic Markov kernel given by a bundle $\pi \colon \gv{\Omega,\pr}
\to \gv{\Omega',\pr'}$, i.e.
\[
  D_{\pi}(\omega,\omega')=
  \begin{cases}
    1, & \pi(\omega)=\omega',\\
    0, & \pi(\omega) \neq \omega'
  \end{cases}
\]
has a canonical Bayesian inverse kernel from the base scheme to the total
scheme:
\[
  K_{\pi}(\omega',\omega)=
  \begin{cases}
    \pr(\omega)/\pr'(\omega'), & \pi(\omega)=\omega',\\
    0, & \text{otherwise}.
  \end{cases}
\]
For fixed $\omega'$, the entries of $K_{\pi}(\omega',-)$ sum to $1$, since
$\pi$ is probability-preserving.  Thus $K_{\pi}$ is a Markov kernel whose
conditional law on the fiber over $\omega'$ is exactly the fiber scheme of
$\pi$.  Its action on observables is
\[
  \sum_{\omega\in\Omega}K_{\pi}(\omega',\omega)X(\omega)
  =
  \frac{1}{\pr'(\omega')}
  \sum_{\omega\in\pi^{-1}(\omega')}X(\omega)\pr(\omega)
  =
  (\pi_{\flat}X)(\omega').
\]
Thus $\pi_{\flat}$ is the observable-side action of this Bayesian inverse.
The identity $\pi_{\flat}\pi^{\sharp}=\id$ says that base-level observables are
unchanged after lifting to the total scheme and averaging back over the
posterior fiber law.  The projection $\pi^{\sharp}\pi_{\flat}$ replaces a
random variable by its fiberwise posterior average, forgetting its variation
inside each fiber.
\end{remark}
\subsection{Bundles, Algebras and Partitions}
\label{ssec:bundles_algebras_partitions}

Any surjection $\pi$ from a scheme $\gv{\Omega, \pr}$ gives the bundle
$\pi \colon \Omega \to [\pi]$.  This association is compatible with function
composition: if $\pi=\pi''\pi'$ as surjections with domain $\gv{\Omega,\pr}$,
then the associated maps, equipped with their induced base distributions,
are bundles and satisfy the same factorization.  We denote by $\pi_{\equiv}$
the bundle associated with the quotient map of an equivalence relation $\equiv$
(or a partition) on $\Omega$.

The association $\pi \mapsto \cA_{\pi}$ depends only on the fiber partition
of $\pi$.  Up to relabeling the base scheme, it is actually one-to-one.  To
construct the bundle $\pi_{\cA}$ whose associated algebra is a given
$\Rr$-subalgebra $\cA$ of $\Rr^{\Omega}$, observe that $\cA$ will be
generated by the indicator functions of the fibers of $\pi_{\cA}$.  Thus,
two points belong to the same fiber of $\pi_{\cA}$ if and only if they are
indistinguishable by members of $\cA$.  Consequently, $\pi_{\cA}$ must be
the quotient bundle of the equivalence relation $\equiv_{\cA}$ defined by
\[
  \omega_1 \equiv_{\cA} \omega_2 \iff X(\omega_1) = X(\omega_2)\ \text{for
  all}\ X \in \cA.
\]
It remains to show that the associated algebra of $\pi_{\cA}$ is actually
$\cA$.  From the definition of $\equiv_{\cA}$, elements of $\cA$ are constant
on the $\equiv_{\cA}$-classes (also known as the \emph{atoms} of $\cA$), so
$\cA$ is contained in the associated algebra of $\pi_{\cA}$.  To establish the
reverse inclusion, it suffices to show that $\cA$ contains the indicator
functions of its atoms (\cite[Chapter~2]{nelson}).  Let $A$ be an atom of
$\cA$ and $\omega \in \Omega \setminus A$.  Then, by the definition of
$\equiv_{\cA}$, there is some $X \in \cA$ whose value at $\omega$ differs from
its common value $X_A$ on $A$.  Since $\cA$ contains the constants, the function
\[
  X_{\omega} := \frac{X - X(\omega)}{X_A-X(\omega)},
\]
which is $1$ on $A$ and $0$ at $\omega$, is in $\cA$.  As a result, $1_A$,
which is the product $\prod_{\omega \notin A} X_{\omega}$, with the empty
product interpreted as $1$, belongs to $\cA$ as well.  To summarize, the
associations
\[
 \pi \mapsto \cA_{\pi},\quad \cA \mapsto \equiv_{\cA},\quad \equiv \mapsto
\pi_{\equiv}
\]
are one-to-one correspondences among three types of objects: (1) bundles
with total scheme $\gv{\Omega,\pr}$, considered up to relabeling of the
base scheme, (2) $\Rr$-subalgebras of $\Rr^{\Omega}$, and (3) partitions
of $\Omega$ (equivalently, equivalence relations on $\Omega$).  The inverse
correspondence for any one of the three descriptions is obtained by composing
the other two.  We leave it to the reader to verify that these correspondences
are order-preserving in the following sense: $\cB$ is an $\Rr$-subalgebra
of $\cA$ if and only if $\equiv_{\cA}$
 refines $\equiv_{\cB}$ and $\equiv$ refines $\equiv'$ if and only if
 $\pi_{\equiv'}$ factors through $\pi_{\equiv}$.

 \begin{example}
   We record the simplest case of these correspondences.
   \label{ex:simplest_corr}
   Let $Z$ be a random function on $\Omega$. Two
   outcomes in $\Omega$ are $\equiv_{\Rr[Z]}$-equivalent if and only if they
   have the same image under $Z$ (and hence under any $X \in \Rr[Z]$). Thus,
   the bundle corresponding to the algebra $\Rr[Z]$ is $\Omega \to [Z]$.
 \end{example}

\section{Conditional Expectations}
\label{sec:cond_expect}
Let $\cA$ be an $\Rr$-subalgebra of random variables on $\gv{\Omega,\pr}$,
and let $\atA$ be the base scheme whose points are the atoms of $\cA$, with
probability inherited from $\gv{\Omega,\pr}$.  Thus
$\pi_{\cA}\colon \Omega \to \atA$ sends each outcome to the atom containing it.
For $X \in \Rr^{\Omega}$, the \emph{conditional expectation of $X$ relative to
$\cA$} is
\begin{equation}
  \E_{\cA}X = \sum \E(X|A) 1_A
    \label{eq:relexp}
\end{equation}
where the sum runs through the atoms $A$ of $\cA$.  Another common notation
for $\E_{\cA}X$ is $\E(X|\cA)$.  We also write $\E(X|X_1, \ldots, X_n)$ for
the conditional expectation of $X$ relative to the $\Rr$-algebra generated
by $X_1, \ldots, X_n$.  When an atom $A$ is used as an argument of a function on
$\atA$, it is being regarded as the corresponding base point.  With this
convention, the conditional average of $X$ on $A$ is
\[
  \E(X|A) = (\pi_{\cA\flat}X)(A)
  = (\pi_{\cA}^{\sharp}\pi_{\cA\flat}X)(\omega)\quad (\omega \in A).
\]
Hence, $\E_{\cA}$ is the operator $\pi_{\cA}^{\sharp}\pi_{\cA\flat}$, which, by
Theorem~\ref{th:flat-sharp}~\eqref{i:orth_proj}, is the orthogonal projection
of $\Rr^{\Omega}$ onto $\cA$.

The key properties of conditional expectations, recognized in various
sources (e.g.,~\cite[Chapter 2]{nelson},~\cite[\S2.3, Theorem]{int_and_prob},
and~\cite[Theorem 5.1]{repas}), are quite transparent from the perspective of
orthogonal projections.  We now demonstrate how to derive them in a functorial
way from Theorem~\ref{th:flat-sharp} without explicitly appealing to geometry.

\begin{proposition}
  Let $\cA$ be an $\Rr$-subalgebra of $\Rr^{\Omega}$. Then
  \begin{enumerate}[i.]
    \item $\E_{\cA}$ is $\cA$-linear and its restriction to $\cA$ is the
      identity map.\label{i:A-linear}
    \item $\E_{\cB}\E_{\cA} = \E_{\cB}$ for any $\Rr$-subalgebra $\cB$
      of $\cA$.\label{i:tower}
    \item $\E'_{\cA}\E_{\cA} = \E$ where $\E'_{\cA}$ is the expectation
      induced by the distribution on the base scheme of
      $\pi_{\cA}$.\label{i:rm}
  \end{enumerate}
  \label{p:rep}
\end{proposition}
\begin{proof}
  We first prove~\eqref{i:A-linear} and~\eqref{i:rm}, and then~\eqref{i:tower}
  the tower property.  Let $\pi = \pi_{\cA} \colon \Omega \to \Omega_{\cA}$
  be the bundle corresponding to $\cA$.  Take any $X \in \Rr^{\Omega}$
  and $Y \in \cA$.  As an element of $\cA$, $Y = \pi^{\sharp}Z$ for
  some random variable $Z$ on $\Omega_{\cA}$.  Since $\pi^{\sharp}$ is
  an algebra map and $\pi_{\flat}$ is an $\Rr^{\Omega_\cA}$-module map
  (Theorem~\ref{th:flat-sharp}~\eqref{i:proj_flat}),
  \begin{align*}
    \E_{\cA}(YX) &= \pi^{\sharp}\pi_{\flat}(YX) =
    \pi^{\sharp}\pi_{\flat}(\pi^{\sharp}Z X) \\ &= \pi^{\sharp}(Z\pi_{\flat}
    X) = \pi^{\sharp}Z \pi^{\sharp}\pi_{\flat}X
  = Y\E_{\cA}X.
  \end{align*}
  The $\cA$-linearity of $\E_{\cA}$ follows, and by taking $X$ to be
  $1_\Omega$, we see that $\E_{\cA}$ acts as the identity map on $\cA$.

  To make sense of the composition $\E_{\cA}'\E_{\cA}$ in~\eqref{i:rm},
  random variables in $\cA$ are identified with those on the base scheme of
  $\pi$ via $\pi_{\flat}$.  By Theorem~\ref{th:flat-sharp}~\eqref{i:section}
  and~\eqref{i:preserve-E}, we have
\[
  \E'_{\cA}\E_{\cA} = \E'_{\cA}\pi_{\flat}\E_{\cA} = \E_{\cA}'
  \pi_{\flat}(\pi^{\sharp} \pi_{\flat}) = \E_{\cA}' (\pi_{\flat}\pi^{\sharp})
  \pi_{\flat} = \E'_{\cA}\pi_{\flat} = \E.
\]
Finally, when $\cB$ is an $\Rr$-subalgebra of $\cA$, the atoms of
$\cB$ are disjoint unions of atoms of $\cA$, so $\pi_{\cB}$ factors as
$\pi_{\cB\cA}\pi_{\cA}$, where $\pi_{\cB\cA}$ is the bundle sending
an atom of $\cA$ to the atom of $\cB$ containing it.  Thus the induced maps
fit into the diagram
\[
\xymatrix@C=5.2em@R=3.2em{
  \Rr^\Omega
    \ar[r]^-{\pi_{\cA\flat}}
    \ar[d]_{\pi_{\cA\flat}}
  &
  \Rr^{\Omega_{\cA}}
    \ar[r]^-{\pi_{\cB\cA\flat}}
  &
  \Rr^{\Omega_{\cB}}
  \ar[d]^{\pi_{\cB}^{\sharp}}
  \\
  \Rr^{\Omega_{\cA}}
    \ar[r]_-{\pi_{\cA}^{\sharp}}
  &
  \Rr^{\Omega}
    \ar[r]_-{\E_{\cB}}
    \ar[u]_-{\pi_{\cA\flat}}
  &
  \Rr^{\Omega}.
}
\]
By Theorem~\ref{th:flat-sharp}~\eqref{i:section},
$\pi_{\cA\flat}\pi_{\cA}^{\sharp} =\id_{\Rr^{\Omega_{\cA}}}$, so the left
square commutes and since $\E_{\cB} = \pi_{\cB}^{\sharp}\pi_{\cB\flat}
= \pi_{\cB}^{\sharp}\pi_{\cB\cA\flat}\pi_{\cA\flat}$ the right square
commutes as well.  Therefore, the outer rectangle commutes, and that means
$\E_{\cB}\E_{\cA} = \E_{\cB}$.
\end{proof}
Notions defined via expectation admit relativized versions by replacing
expectation by conditional expectation.  For example, the \emph{covariance
of $X$ and $Y$ relative to an $\Rr$-algebra $\cA$} of random variables is
\[
  \cov_{\cA}(X,Y)
    := \E_{\cA}\bigl((X-\E_{\cA}X)(Y-\E_{\cA}Y)\bigr).
\]
Consequently, results expressed in terms of expectation, such as H\"older's
inequality, Jensen's inequality, and Chebyshev's inequality, all have
relativized versions; see~\cite[Chapter 2]{nelson} for details.
The expectation $\E$ on $\Rr^{\Omega}$ is just the conditional expectation
$\E_{\Rr}$.  Therefore, any relativized notion reduces to the standard one
when the algebra is $\Rr$.

\section{Orthogonality} \label{sec:orth}

To demonstrate the strength of Proposition~\ref{p:rep}, we derive several
standard results from elementary probability and statistics.  While the
orthogonality arguments underlying these proofs are well-known to experts,
we hope this perspective proves illuminating for readers who are less familiar
with it.

First, the \emph{law of total expectation} follows as a special case of the
identity $\E = \E_{\cA}'\E_{\cA}$ in Proposition~\ref{p:rep}.  By setting
$\cA = \Rr[X]$ to be the $\Rr$-algebra generated by a random variable $X$,
we obtain for any $Y \in \Rr^{\Omega}$:
\begin{equation}
  \begin{split}
    \E{Y} &= \E_{\Rr[X]}'\E_{\Rr[X]}Y = \E'_{\Rr[X]} \sum_{x \in X(\Omega)}
    \E(Y|X=x)1_{X = x}\\ &= \sum_{x \in X(\Omega)}
    \E(Y|X=x)\pr_X(x) = \E\E(Y|X).
\end{split}
  \label{eq:lote}
\end{equation}
When $Y$ is a function of $X$, Equation~\eqref{eq:lote} yields the \emph{law
of the unconscious statistician}:
\begin{equation}
  \E f(X) = \sum_{x \in X(\Omega)} \E(f(X)|X=x)\pr_X(x) = \sum_{x \in
  X(\Omega)} f(x)\pr_X(x).\label{eq:LOTUS}
\end{equation}
For a partition $\{A_1, \ldots, A_n\}$ of $\Omega$, take $X \in \Rr^{\Omega}$
to be the random function defined by $X(\omega) = i$ for $\omega\in A_i$.  The
law of total expectation then becomes
\begin{equation}
\E{Y} = \sum_{i=1}^n \E(Y|A_i)\Pr(A_i).\label{eq:s_lote}
\end{equation}
For an event $B \subseteq \Omega$ and a $\Rr$-subalgebra $\cA$ of
$\Rr^{\Omega}$, the \emph{conditional probability of $B$ relative to $\cA$}
is $\Pr_{\cA}(B):=\E_{\cA}1_B$.  Observe that if $A$
is an atom of $\cA$, then the value of $\Pr_{\cA}(B)$ at any $\omega \in A$ is
\[
  \E(1_B|A) = \frac{\E(1_B1_A)}{\E(1_A1_A)} =
  \frac{\Pr(B\cap A)}{\Pr(A)} = \Pr(B|A).
\]
Thus, substituting $Y=1_B$ into Equation~\eqref{eq:s_lote} yields the \emph{law
of total probability}:
\begin{equation}
  \Pr(B) = \E 1_B = \sum_{i=1}^n \E(1_B|A_i)\Pr(A_i) = \sum_{i=1}^n
  \Pr(B|A_i)\Pr(A_i).
  \label{eq:ltprob}
\end{equation}

Let $\cB \subseteq \cA$ be $\Rr$-subalgebras of $\Rr^{\Omega}$.  For random
variables $X_1, X_2 \in \Rr^{\Omega}$, consider the orthogonal decomposition
with respect to $\cA$:
\begin{align*}
  X_i-\E_{\cB}X_i &= (X_i-\E_{\cA}X_i) + (\E_{\cA}X_i - \E_{\cB}X_i), \qquad
  i \in \{1,2\}.
\end{align*}
Since $X_i-\E_{\cA}X_i \in \ker\E_{\cA}$ and $\E_{\cA}X_i - \E_{\cB}X_i \in
\cA$, the $\cA$-linearity of $\E_{\cA}$ gives
\begin{align*}
  \E_{\cA}(X_1 - \E_{\cB}X_1)(X_2-\E_{\cB}X_2)
  &=\E_{\cA}(X_1-\E_{\cA}X_1)(X_2-\E_{\cA}X_2)\\
  &+ (\E_{\cA}X_1 - \E_{\cB}X_1)(\E_{\cA}X_2 - \E_{\cB}X_2).
\end{align*}
Projecting further down to $\cB$ and using the fact $\E_{\cB}\E_{\cA} =
\E_{\cB}$, we get
\begin{align*}
  \E_{\cB}(X_1-\E_{\cB}X_1)(X_2-\E_{\cB}X_2) &
  = \E_{\cB}\E_{\cA}(X_1-\E_{\cA}X_1)(X_2-\E_{\cA}X_2) \\ &+
  \E_{\cB}(\E_{\cA}X_1 - \E_{\cB}X_1)(\E_{\cA}X_2 - \E_{\cB}X_2).
\end{align*}
In other words,
\begin{equation}
  \cov_{\cB}(X_1,X_2) = \E_{\cB}\cov_{\cA}(X_1,X_2) +
  \cov_{\cB}(\E_{\cA}X_1,\E_{\cA}X_2).  \label{eq:lotc}
\end{equation}
When $X_1=X_2=X$, Equation~\eqref{eq:lotc} becomes
\begin{equation}
  \var_{\cB}X = \E_{\cB}\var_{\cA}X + \var_{\cB} \E_{\cA}X.\label{eq:loftv}
\end{equation}
When $\cB=\Rr$ and $\cA=\Rr[Z]$, Equations~\eqref{eq:lotc}
and~\eqref{eq:loftv} become
\[
  \cov(X_1,X_2)=\E(\cov(X_1,X_2|Z))+\cov(\E(X_1|Z),\E(X_2|Z))
\]
and
\[
  \var X = \E(\var(X|Z)) + \var(\E(X|Z)).
\]
They are the usual \emph{law of total covariance} and the \emph{law of
variance}, respectively.

Setting $\cB = \Rr$ in Equation~\eqref{eq:loftv} and using the facts that
$\E\E_{\cA} = \E = \E'_{\cA}\E_{\cA}$ and $\E_{\cA}$ acts as the identity
on $\cA$ (Proposition~\ref{p:rep}), we get
\begin{equation}
  \begin{split}
    \var X &= \E \var_{\cA} X + \var \E_{\cA}X \\
    &= \E\E_{\cA}(X-\E_{\cA}X)^2
    + \E(\E_{\cA}X - \E\E_{\cA}X)^2 \\
    &= \E(X-\E_{\cA}X)^2
    + \E'_{\cA}\E_{\cA}(\E_{\cA}X - \E{X})^2 \\
    &= \E(X-\E_{\cA}X)^2
    + \E'_{\cA}(\E_{\cA}X - \E{X})^2.
  \end{split}
  \label{eq:covf}
\end{equation}
When $\cA$ is the algebra corresponding to the partition of the uniform scheme
$\Omega[nm]$ into $m$ groups of size $n$, Equation~\eqref{eq:covf} becomes
\begin{equation*}
  \frac{\sum_{\omega} (X(\omega)-\E{X})^2}{nm} = \frac{\sum_{\omega}
  (X(\omega) - \E(X|A_{\omega}))^2}{nm} + \frac{\sum_{A} (\E(X|A) -
  \E{X})^2}{m},
\end{equation*}
where $\omega$ runs through the elements of $\Omega$, $A$ runs through the
groups, and $A_{\omega}$ denotes the group containing $\omega$.  Clearing
denominators yields
\begin{equation*}
  \sum_{\omega} (X(\omega)-\E{X})^2
  = \sum_{\omega} (X(\omega)-\E(X|A_{\omega}))^2
  + n\sum_{A} (\E(X|A) - \E{X})^2.
\end{equation*}
This gives the \emph{components of variance formula}, which expresses the
total sum of squares (TSS) as the sum of the within-group sum of squares
(WSS) and the usual between-group sum of squares (BSS).

As an illustration of Nelson's quotation at the beginning of this article,
we give a short proof of a fibered version of the \emph{weak law of large
numbers}.  The argument is the usual Chebyshev proof, in the form suggested
by an exercise in Sternberg's notes on real analysis~\cite[Problem~2,
Problem Set~1]{ssreal}.  Let $X_1,X_2,\ldots$ be random variables on
a scheme $\gv{\Omega,\pr}$, and let $\cA$ be an $\Rr$-subalgebra of
$\Rr^{\Omega}$.  Suppose that the $X_i$ are pairwise uncorrelated relative
to $\cA$; that is, $\cov_{\cA}(X_i,X_j)=0$ for $i\ne j$.  Suppose also that
there is an element $K\in\cA$ such that $\var_{\cA}X_i\le K$ for all $i$, in
the pointwise order on $\cA$.  Then the conditional variance of the empirical
average $\bar{X}_n:=(1/n)\sum_{i=1}^n X_i$ has the following bound:
\begin{equation}
  \var_{\cA}\bar X_n
  = \frac{1}{n^2}\sum_{i,j=1}^n \cov_{\cA}(X_i,X_j)
  = \frac{1}{n^2}\sum_{i=1}^n \var_{\cA}X_i
  \le \frac{K}{n}.
  \label{eq:var-empr-mean}
\end{equation}

The relative version of Chebyshev's inequality~\cite[Chapter~2]{nelson}
converts this estimate into a conditional probability estimate: for
$\varepsilon > 0$,
\begin{equation}
  \Pr\nolimits_{\cA} \left\{|\bar X_n-\E_{\cA}\bar X_n|\ge \varepsilon\right\}
  \le \frac{\var_{\cA}\bar X_n}{\varepsilon^2} \le \frac{K}{n\varepsilon^2}.
  \label{eq:bundle-WLLN}
\end{equation}
In particular, if the $X_i$'s have a common conditional expectation relative
to $\cA$, equivalently if their orthogonal projections onto $\cA$ are all the
same element $\mu\in\cA$, then $\E_{\cA}\bar X_n=\mu$ and
\begin{equation}
  \Pr\nolimits_{\cA}\left\{|\bar X_n-\mu|\ge\varepsilon\right\} \le
  \frac{K}{n\varepsilon^2}.  \label{eq:WLLN-v2}
\end{equation}
Thus the conditional probability of deviation from the common conditional mean
tends to $0$ pointwise on each atom of $\cA$.  This is the relative finite-space
form of the weak law of large numbers.

Orthogonality has many other applications in probability theory and
statistics.  For example, it gives the usual geometry of least-squares
regression, where fitted values are orthogonal projections and the Pearson
correlation coefficient and coefficient of determination are cosine and
squared-projection quantities.  These projections are usually onto linear
subspaces rather than subalgebras, so they do not directly arise from bundles.
The interested reader may consult~\cite{hspm,rodgers-nicewander-correlation};
for a broader Hilbert-space treatment, see~\cite{hspm,HSM}.

\section{Fiber Products}
\label{sec:fprod}

Fiber products generalize independent products of schemes to arbitrary
bases.  Recall that the independent product of two schemes $\gv{\Omega,
\pr}$ and $\gv{\Omega',\pr'}$ is the scheme on $\Omega \times \Omega'$
with distribution
\[
  \gv{\omega,\omega'} \mapsto \pr(\omega)\pr'(\omega').
\]
Now fix two bundles over the same base scheme,
\[
  \gv{\Omega_1, \pr_1} \xrightarrow{\pi_1} \gv{\Omega_0,\pr_0}
\xleftarrow{\pi_2} \gv{\Omega_2,\pr_2}
\]
and define their \emph{fiber product} $\fprod$ as follows.  Its underlying
set is the set-theoretic fiber product
\begin{equation}
  \begin{split}
  \xymatrix{
    &\sfprod \ar[d]_{\theta_1} \ar[r]^(.6){\theta_2} \ar[rd]^{\pi} & \Omega_2
    \ar[d]^{\pi_2} \\ &\Omega_1 \ar[r]_{\pi_1} & \Omega_0
  }
  \end{split}
  \label{eq:fprod}
\end{equation}
where $\theta_i$ ($i=1,2)$ is the coordinate projection to $\Omega_i$,
and $\pi=\pi_1\theta_1=\pi_2\theta_2$.  The distribution $\pr_{\otimes}$ on
$\fprod$ is forced by requiring the fiber scheme of $\pi$ over each $\omega_0
\in\Omega_0$ to be the independent product of the corresponding fiber schemes
of $\pi_{1\omega_0}$ and $\pi_{2\omega_0}$.  If $\gv{\omega_1,\omega_2}
\in \Omega_1\times_{\Omega_0}\Omega_2$ lies over $\omega_0$, then the
fiberwise-independence requirement translates to
\[
  \frac{\pr_{\otimes}(\gv{\omega_1,\omega_2})}{\pr_0(\omega_0)}
  = \frac{\pr_1(\omega_1)}{\pr_0(\omega_0)}
  \frac{\pr_2(\omega_2)}{\pr_0(\omega_0)}.
\]
Equivalently,
\begin{equation}
  \frac{\pr_{\otimes}}{\pi^{\sharp}\pr_0} =
  \theta_1^{\sharp}\left(\frac{\pr_1}{\pi_1^{\sharp}\pr_0}\right)
  \theta_2^{\sharp}\left(\frac{\pr_2}{\pi_2^{\sharp}\pr_0}\right).
  \label{eq:fiberwise-ind}
\end{equation}
Since $\theta_1^\sharp\pi_1^\sharp =\theta_2^\sharp\pi_2^\sharp =
\pi^{\sharp}$, that means $\pr_{\otimes}$ must be defined as
\begin{equation}
  \pr_{\otimes} := \frac{(\theta_1^{\sharp}\pr_1)(\theta_2^{\sharp}\pr_2)
  }{\pi^{\sharp}\pr_0} \label{def:pr_fb}
\end{equation}
A straightforward computation shows that $\pr_{\otimes}$ thus defined is
indeed a distribution on $\sfprod$ making the diagram in~\eqref{eq:fprod}
commutative as a square of bundles.  The key step is to establish the
base-change identity, for which we give a conceptual explanation in
Section~\ref{sec:transfer-Beck_Chevalley}:
\begin{equation}
  \theta_{1*}\theta_2^{\sharp}=\pi_1^{\sharp}\pi_{2*}.
  \label{eq:base-change}
\end{equation}
Indeed, for $a_2\in\Rr^{\Omega_2}$ and $\omega_1\in\Omega_1$,
\[
\begin{split}
  (\theta_{1*}\theta_2^\sharp a_2)(\omega_1) &=
  \sum_{\substack{(\eta_1,\eta_2)\in
      \Omega_1\times_{\Omega_0}\Omega_2\\ \eta_1=\omega_1}} a_2(\eta_2) =
      \sum_{\eta_2:\,\pi_2(\eta_2)=\pi_1(\omega_1)} a_2(\eta_2) \\ &=
  (\pi_{2*}a_2)(\pi_1(\omega_1)) = (\pi_1^\sharp\pi_{2*}a_2)(\omega_1).
\end{split}
\]
Since the function $\pr_{\otimes}$ defined in~\eqref{def:pr_fb}
is clearly positive, it remains to check that it is normalized
and that $\theta_1$ is probability-preserving.  To do so, it
suffices to show that $\theta_{1*}\pr_{\otimes}=\pr_1$.  Using the
base-change identity~\eqref{eq:base-change} and the projection
formula~\eqref{eq:proj_formula} for $\theta_1$, we get
  \[
    \begin{split}
      \theta_{1*}\pr_{\otimes}
      &=\theta_{1*}\left(
        \theta_1^{\sharp}\left(\frac{\pr_1}{\pi_1^{\sharp}\pr_0}\right)
        \theta_2^{\sharp}\pr_2\right)\\
      &=\frac{\pr_1}{\pi_1^{\sharp}\pr_0}\,
        \theta_{1*}\theta_2^{\sharp}\pr_2
       =\frac{\pr_1}{\pi_1^{\sharp}\pr_0}\,
        \pi_1^{\sharp}\pi_{2*}\pr_2
       =\pr_1.
    \end{split}
  \]
  The same argument gives $\theta_{2*}\pr_{\otimes}=\pr_2$.  Hence
  both coordinate projections are bundles.  Consequently,
  $\pi=\pi_1\theta_1=\pi_2\theta_2$ is also a bundle.

For fixed $\Omega_0$, the evident symmetry
$\Omega_1\otimes_{\Omega_0}\Omega_2\cong \Omega_2\otimes_{\Omega_0}\Omega_1$,
the identity bundle $\Omega_0 \to \Omega_0$ as unit, and the associativity
proved in Corollary~\ref{cor:fiber-product-associative} below make the
category of bundles over $\Omega_0$ into a symmetric monoidal category.  In
particular, ordinary independence is the special case in which $\Omega_0 =
\bullet$, so the ordinary independent product is the fiber product in the
slice over the terminal scheme.

Using fiber product, we can show that the category of bundles does not admit a
categorical product.  Suppose, for the sake of contradiction, that it does,
and let $\tau_1,\tau_2\colon P\to\Omega T$ be the product of a uniform
scheme $\Omega T$ with itself.  Since $\Omega T\otimes_{\bullet}\Omega
T\cong\Omega(T^2)$ has coordinate projections to $\Omega T$, the product
property would give a unique bundle $h\colon \Omega(T^2)\to P$ with $\tau_i
h=\pi_i$ for $i=1,2$.  Since the coordinate pair $\gv{\pi_1,\pi_2}$ separates
points, $h$ is injective.  As a bundle, $h$ is also surjective, so $h$ must
be an isomorphism.  Thus $P$ is isomorphic to the uniform scheme $\Omega(T^2)$.

The schemes $[1_A]$ and $[1_B]$ in Example~\ref{ex:distr-sch} are both
$\Omega T$ for $T=\{0,1\}$, and $[1_A,1_B]$ has coordinate bundles to
them.  Applying the same argument to these coordinate bundles would make
$[1_A,1_B]$ isomorphic to $P$ and hence isomorphic to the uniform scheme on
$\{0,1\}^2$.  However, this is clearly impossible since the four outcomes in
$[1_A,1_B]$ have probabilities $1/3,1/6,1/6,1/3$, not $1/4$.

Although the fiber product of bundles is not the categorical product of the
category of bundles, the stack of two fiber products is still a fiber product:

\begin{proposition}
  The outer rectangle obtained by pasting two fiber products of bundles
  as shown
  \begin{equation}
    \begin{split}
  \xymatrix{
    \Omega_3 \otimes_{\Omega_2} (\Omega_2 \otimes_{\Omega_0}\Omega_1)
      \ar[r] \ar[d] & \Omega_2 \otimes_{\Omega_0} \Omega_1 \ar[r] \ar[d]
      & \Omega_1 \ar[d]^{\alpha} \\ \Omega_3 \ar[r]_{\gamma} & \Omega_2
      \ar[r]_{\beta} & \Omega_0
  }
\end{split}
 \end{equation}
 is isomorphic to the fiber product of the bundles $\alpha$ and $\delta =
 \beta\gamma$.
  \label{p:stack_fprod}
\end{proposition}
\begin{proof}
  Since set-theoretic fiber products are pullbacks, it follows from the pasting
  law of pullbacks (see~\cite[Lemma 5.8]{awodey-category-theory}) that the map
  \[
    \gv{\omega_3,\omega_1} \mapsto
    \gv{\omega_3,\gv{\gamma(\omega_3),\omega_1}}
  \]
  gives the isomorphism from the set-theoretic fiber $\Omega_3
  \times_{\Omega_0} \Omega_1$ to the diagram obtained from the outer rectangle
  by forgetting the distribution.  It remains to check that this map preserves
  the distributions.  Put
  \[
    \omega_2=\gamma(\omega_3),\qquad
    \omega_0=\alpha(\omega_1)=\beta(\omega_2)=\delta(\omega_3).
  \]
  The probability of $\gv{\omega_3,\omega_1}$ in $\Omega_3
  \otimes_{\Omega_0} \Omega_1$ is
  \[
    \frac{\pr_3(\omega_3)\pr_1(\omega_1)}{\pr_0(\omega_0)}
  \]
  The point $\gv{\omega_2,\omega_1}$ has probability
  \[
    \frac{\pr_2(\omega_2)\pr_1(\omega_1)}{\pr_0(\omega_0)}
  \]
  in $\Omega_2\otimes_{\Omega_0}\Omega_1$.  Hence the probability of
  the image $\gv{\omega_3,\gv{\omega_2,\omega_1}}$ in the fiber product
  $\Omega_3\otimes_{\Omega_2}(\Omega_2\otimes_{\Omega_0}\Omega_1)$ is
  \[
    \frac{\pr_3(\omega_3)}{\pr_2(\omega_2)}
    \left(\frac{\pr_2(\omega_2)\pr_1(\omega_1)}{\pr_0(\omega_0)}\right)
    =
    \frac{\pr_3(\omega_3)\pr_1(\omega_1)}{\pr_0(\omega_0)}.
  \]
  The bijection therefore preserves point masses, and hence is an isomorphism
  of schemes.
\end{proof}

\begin{corollary}
  Fiber products of bundles are associative.  More precisely, the rebracketing
  map
  \[
    \Omega_1\otimes_{\Gamma_0}(\Omega_2\otimes_{\Gamma_1}\Omega_3)
    \longrightarrow
    (\Omega_1\otimes_{\Gamma_0}\Omega_2)\otimes_{\Gamma_1}\Omega_3
  \]
  given by
  \[
    \gv{\omega_1,\gv{\omega_2,\omega_3}}\longmapsto
    \gv{\gv{\omega_1,\omega_2},\omega_3}
  \]
  is an isomorphism of schemes.
  \label{cor:fiber-product-associative}
\end{corollary}
\begin{proof}
  Taking fiber products yields the diagram
  
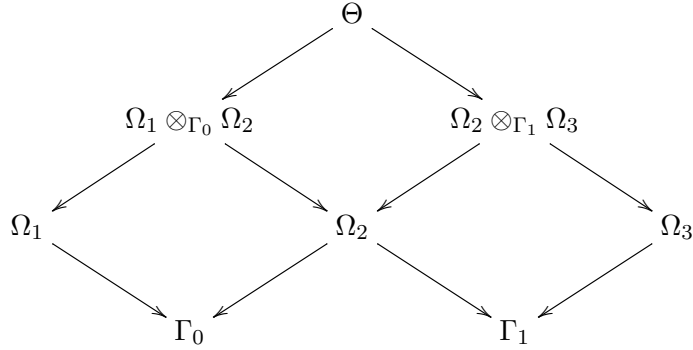
\begin{equation*}
    \xymatrix{ & & &\Theta \ar[dl] \ar[dr] & & \\
      & &\Omega_1 \otimes_{\Gamma_0} \Omega_2 \ar[dl] \ar[dr] & &\Omega_2
      \otimes_{\Gamma_1} \Omega_3 \ar[dl] \ar[dr] & \\
      &\Omega_1 \ar[dr] & &\Omega_2 \ar[dl] \ar[dr] & &\Omega_3 \ar[dl] \\
      & &\Gamma_0 & &\Gamma_1 &
   }
  \end{equation*}
  where
  \[
    \Theta:=(\Omega_1\otimes_{\Gamma_0}\Omega_2)\otimes_{\Omega_2}
    (\Omega_2\otimes_{\Gamma_1}\Omega_3).
  \]
  Apply Proposition~\ref{p:stack_fprod} to the composable pair
  \[
    \Omega_1\otimes_{\Gamma_0}\Omega_2 \longrightarrow \Omega_2
    \longrightarrow \Gamma_1
  \]
  and the bundle $\Omega_3\to\Gamma_1$.  This identifies
  $\Theta$ with
  $(\Omega_1\otimes_{\Gamma_0}\Omega_2)\otimes_{\Gamma_1}\Omega_3$ by sending
  \[
    \gv{\gv{\omega_1,\omega_2},\gv{\omega_2,\omega_3}}
    \longmapsto \gv{\gv{\omega_1,\omega_2},\omega_3}.
  \]

  Similarly, after using the symmetry of the two fiber-product factors over
  $\Omega_2$ and the symmetry
  $\Omega_1\otimes_{\Gamma_0}\Omega_2\cong
  \Omega_2\otimes_{\Gamma_0}\Omega_1$, the same scheme $\Theta$ is
  \[
    (\Omega_2\otimes_{\Gamma_1}\Omega_3)\otimes_{\Omega_2}
      (\Omega_2\otimes_{\Gamma_0}\Omega_1).
  \]
  Applying Proposition~\ref{p:stack_fprod} to
  \[
    \Omega_2\otimes_{\Gamma_1}\Omega_3 \longrightarrow \Omega_2
    \longrightarrow \Gamma_0
  \]
  and the bundle $\Omega_1\to\Gamma_0$ identifies this scheme with
  $(\Omega_2\otimes_{\Gamma_1}\Omega_3)\otimes_{\Gamma_0}\Omega_1$.
  The symmetry over $\Gamma_0$ then identifies it with
  $\Omega_1\otimes_{\Gamma_0}(\Omega_2\otimes_{\Gamma_1}\Omega_3)$, carrying
  \[
    \gv{\gv{\omega_1,\omega_2},\gv{\omega_2,\omega_3}}
    \longmapsto \gv{\omega_1,\gv{\omega_2,\omega_3}}.
  \]
  Composing these identifications gives exactly the rebracketing map displayed
  in the statement.
\end{proof}

The probabilistic content of fiber products of bundles is conditional
independence.  Let $X,Y,Z$ be random functions on a common scheme.
The distribution schemes $[\gv{X,Z}]$ and $[\gv{Y,Z}]$ both map to $[Z]$.
The canonical map
\[
  [\gv{X,Y,Z}] \longrightarrow [\gv{X,Z}]\otimes_{[Z]}[\gv{Y,Z}]
\]
on underlying sets is given by the universal property of set-theoretic fiber
products; it sends
$\gv{x,y,z}$ to $\gv{\gv{x,z},\gv{y,z}}$ and is injective.  The following
observation is often the most economical way to read a fiber product of
probability schemes.

\begin{proposition}
  The random functions $X$ and $Y$ are conditionally independent given $Z$ if
  and only if the natural injection
  \[
    [\gv{X,Y,Z}] \longrightarrow [\gv{X,Z}]\otimes_{[Z]}[\gv{Y,Z}]
  \]
  is an isomorphism of schemes.
  \label{p:cond-ind-fb}
\end{proposition}
\begin{proof}
  Fix a point $\gv{\gv{x,z},\gv{y,z}}$ of the fiber product; equivalently,
  $\Pr(X=x,Z=z)>0$ and $\Pr(Y=y,Z=z)>0$.  The probability assigned by the
  fiber product to
  $\gv{\gv{x,z},\gv{y,z}}$ is
  \begin{align*}
    \frac{\Pr(X=x,Z=z)\Pr(Y=y,Z=z)}{\Pr(Z=z)},
  \end{align*}
  i.e. $\Pr(Z=z)\Pr(X=x|Z=z)\Pr(Y=y|Z=z)$.  The joint probability of the
  corresponding event is
  \[
    \Pr(X=x,Y=y,Z=z)=\Pr(Z=z)\Pr(X=x,Y=y|Z=z).
  \]
  Thus, if the canonical map is an isomorphism of schemes, then the two
  distributions are equal, which means that
  \[
    \Pr(X=x|Z=z)\Pr(Y=y|Z=z) = \Pr(X=x,Y=y|Z=z).
  \]
  Therefore, $X$ and $Y$ are conditionally independent given $Z$. Conversely,
  assume conditional independence.  For every point $\gv{\gv{x,z},\gv{y,z}}$ in
  the underlying set of the fiber product $[\gv{X,Z}]\otimes_{[Z]}[\gv{Y,Z}]$,
  the product
  \[
    \Pr(Z=z)\Pr(X=x|Z=z)\Pr(Y=y|Z=z)
  \]
  is positive and, by conditional independence, is equal to
  $\Pr(X=x,Y=y,Z=z)$.  Hence $\gv{x,y,z}$ lies in the range of
  $\gv{X,Y,Z}$.  Every point of the target is therefore attained by the
  canonical injection, so the injection is surjective; the same equality of
  probabilities shows that it preserves the distributions.
\end{proof}

Iterated fiber products (well-defined up to isomorphism by associativity
of fiber product) therefore provide a compact way to encode repeated
conditional-independence gluings.  We close by applying this viewpoint to
finite Markov chains.  A sequence $X_1, \ldots, X_n$ ($n \ge 2$) of random
elements on a scheme $\gv{\Omega,\pr}$ is a (\emph{discrete-time}) \emph{Markov
chain} if it satisfies the \emph{Markov property}: for each $2 \le i \le n$
and each $\gv{x_1, \ldots, x_i}$ in the range of $\gv{X_1, \ldots, X_i}$,
\begin{equation*}
  \Pr(X_i=x_i | X_1=x_1, \ldots, X_{i-1}=x_{i-1})
  = \Pr(X_i=x_i | X_{i-1}=x_{i-1}).
\end{equation*}
Equivalently, $\Pr(X_1=x_1, \ldots, X_i=x_i)$ equals
\begin{equation*}
  \Pr(X_1=x_1,\ldots, X_{i-1}=x_{i-1})
  \frac{\Pr(X_{i-1}=x_{i-1}, X_i=x_i)}
       {\Pr(X_{i-1}=x_{i-1})}.
\end{equation*}
By Proposition~\ref{p:cond-ind-fb}, this is equivalent to the assertion that
for each $2 \le i \le n$,
\begin{equation*}
  [\gv{X_1,\ldots,X_i}]
  \cong
  [\gv{X_1,\ldots,X_{i-1}}]\otimes_{[X_{i-1}]}[\gv{X_{i-1},X_i}].
\end{equation*}
Thus, the sequence $X_1, \ldots, X_n$ is a discrete-time Markov chain if
and only if the joint distribution scheme of the chain is isomorphic to the
iterated fiber product of its adjacent-pair schemes:
\begin{equation}
  [\gv{X_1,\ldots,X_n}]
  \cong
  [\gv{X_1,X_2}]\otimes_{[X_2]} \cdots
  \otimes_{[X_{n-1}]}[\gv{X_{n-1},X_n}].
  \label{eq:fb_MC}
\end{equation}
This fiber-product formulation separates a finite Markov model into two
ingredients.  The schemes $[\gv{X_i,X_{i+1}}]$ record the adjacent-pair
distributions, while the iterated fiber product records the rule for gluing
these adjacent pairs by conditional independence over the shared one-variable
schemes $[X_i]$.
It also shows immediately that the reverse sequence $X_n,\ldots,X_1$ is again a
discrete-time Markov chain, since the same iterated fiber product can be read
from right to left.

The framework presented here carries standard finite-probability calculations
functorially.  The point is not that the results above are new, but that the
bundle category makes their shared structure visible.  The constructions
above give a finite categorical dictionary for several standard operations
in probability.  Quotients of schemes, finite algebras of random variables,
and partitions describe the same data; pullback and fiberwise averaging
give conditional expectation; orthogonal projection explains the variance
identities; and fiber products encode the conditional-independence gluings
used in Markov-chain models.

\appendix
\section{Transfer as Trace and Beck--Chevalley}
\label{sec:transfer-Beck_Chevalley}

This appendix records a more conceptual framework for the transfer maps
$\pi_*$ and the base-change identity~\eqref{eq:base-change}, using the
categorical trace formalism in~\cite{ponto-shulman-traces}.  This approach
does not add any probabilistic content, but it explains why the projection
formula and the base-change identity are instances of familiar categorical
compatibilities.

Let $\pi\colon \gv{\Omega,\pr}\to\gv{\Omega',\pr'}$ be a bundle, and write
\[
  R=\Rr^\Omega,\qquad R'=\Rr^{\Omega'}.
\]
The pullback $\pi^\sharp\colon R'\to R$ makes $R$ a finite projective
$R'$-module.  With $\otimes_{R'}$ as the monoidal product and $R'$ as the unit,
the category of finite projective $R'$-modules is symmetric monoidal.  The
dual of $R$ in this category is
\[
  R^\vee=\Hom_{R'}(R,R').
\]
The evaluation map is
\[
  \ev\colon R^\vee\otimes_{R'}R\to R',
  \qquad
  f\otimes Y\longmapsto f(Y).
\]
Since $R$ is finite projective over $R'$, choose a finite dual basis
$Y_i\in R$, $f_i\in R^\vee$, so that
\[
	  Y=\sum_i Y_i f_i(Y)
  \qquad\text{for all }Y\in R.
\]
The coevaluation map is
\[
  \coev\colon R'\to R\otimes_{R'}R^\vee, \qquad 1\longmapsto \sum_i
  Y_i\otimes f_i.
\]
This formula is independent of the chosen dual basis.  Indeed, since $R$ is
finite projective over $R'$, the canonical map
\[
  R \otimes_{R'}R^\vee\longrightarrow \operatorname{End}_{R'}(R), \qquad Y
  \otimes f\longmapsto (X\mapsto Yf(X)),
\]
is an isomorphism.  The element $\sum_i Y_i \otimes f_i$ maps to
$\operatorname{id}_R$, by the dual-basis identity above, and hence is
the canonical preimage of $\operatorname{id}_R$.

The categorical trace of an $R'$-linear map $L \colon R \to R$ is the composite
\[
  R' \xrightarrow{\coev} R\otimes_{R'}R^\vee \xrightarrow{L\otimes 1}
  R\otimes_{R'}R^\vee \xrightarrow{\tau} R^\vee\otimes_{R'}R \xrightarrow{\ev}
  R',
\]
where $\tau(Y\otimes f)=f \otimes Y$ is the twist.  Thus,
\begin{equation}
  \Tr_{R/R'}(L)=\sum_i f_i(LY_i)\in R'.
  \label{eq:trace_formula}
\end{equation}
For the transfer map, the relevant endomorphism is multiplication by a random
variable $X\in R$.  One should expect its relative trace to be fiberwise
summation: over a fixed point $\omega'\in\Omega'$, the $R'$-module $R$ restricts
to the coordinate functions on the finite fiber $\pi^{-1}(\omega')$, and
multiplication by $X$ is diagonal there with diagonal entries
$X(\omega)$, $\omega\in\pi^{-1}(\omega')$.  Taking the trace over the base
therefore amounts to summing these diagonal entries separately on each fiber.
Choose $Y_{\omega} = 1_{\omega}$ and $f_{\omega}(Y) =
Y(\omega)1_{\pi(\omega)}$ ($\omega \in \Omega$) as the dual basis. Take $L$
in Equation~\eqref{eq:trace_formula} to be the multiplication-by-$X$ map, then
at $\omega' \in \Omega'$,
\[
\begin{split}
  \Tr_{R/R'}(m_X)(\omega') &= \sum_{\omega \in \Omega} f_{\omega}(m_X
  1_\omega) = \sum_{\omega \in \Omega} X(\omega)1_{\pi(\omega)}(\omega') \\
  &= \sum_{\omega \in \pi^{-1}(\omega')}X(\omega) = \pi_*X(\omega').
\end{split}
\]
This verifies that the transfer of $X$ is the relative trace of
$m_X$.  The Euler characteristic of the $R'$-module $R$, as defined
in~\cite[Definition~2.2]{ponto-shulman-traces}, is the relative trace of
$\id_R = m_{1_\Omega}$; that is the fiber-size counting function:
\[
  \omega' \mapsto \sum_{\omega \in \pi^{-1}(\omega')} 1 = |\pi^{-1}(\omega')|.
\]
Hence, in this context, the Euler characteristic of a finite-dimensional
vector space coincides with its dimension.

The projection formula~\eqref{eq:proj_formula} is Frobenius reciprocity in
this language.  For $Z\in R'$ and $X\in R$, taking the relative trace of
the $R'$-endomorphism $m_{(\pi^\sharp Z)X}=(\pi^\sharp Z)m_X$ yields
\begin{equation}
\begin{split}
  \pi_*((\pi^\sharp Z)X) &= \Tr_{R/R'}(m_{(\pi^\sharp Z)X}) =
  \Tr_{R/R'}(m_{\pi^\sharp Z}m_X)\\
  &= \ev\tau\left(m_{\pi^{\sharp}Z}m_X
  \otimes 1 \right)\coev\\
  &= \ev(1 \otimes m_{\pi^{\sharp}Z})\tau(m_X \otimes
  1)\coev \\
  &= Z\ev\tau(m_X \otimes 1)\coev = Z\,\Tr_{R/R'}(m_X) = Z\,\pi_*X.
\end{split}
\label{eq:relative_trace}
\end{equation}
Thus $\pi_*$ is $R'$-linear when $R$ is regarded as an $R'$-module through
$\pi^\sharp$.  This is the usual Frobenius reciprocity compatibility between
restriction $\pi^\sharp$ and transfer $\pi_*$.

The base-change identity~\eqref{eq:base-change} has the corresponding
Beck--Chevalley interpretation.  Consider a fiber-product square.  Since the
underlying set of the fiber product is $\Omega_1\times_{\Omega_0}\Omega_2$,
applying the pullback functor $\sharp$ gives the pushout of finite product
algebras
\[
  \begin{array}{c@{\qquad\qquad}c}
  \xymatrix{
    \Omega_1\otimes_{\Omega_0}\Omega_2
      \ar[r]^(.65){\theta_2} \ar[d]_{\theta_1}
      & \Omega_2 \ar[d]^{\pi_2}\\
    \Omega_1 \ar[r]_{\pi_1} & \Omega_0
  }
  &
  \xymatrix{
    R_{12} & R_2 \ar[l]_{\theta_2^\sharp}\\
    R_1 \ar[u]^{\theta_1^\sharp}
      & R_0 \ar[l]^{\pi_1^\sharp} \ar[u]_{\pi_2^\sharp}
  }
  \end{array}
\]
where $R_i=\Rr^{\Omega_i}$ and
\[
  R_{12}\cong R_1\otimes_{R_0}R_2,
  \qquad
  \theta_1^\sharp X_1\,\theta_2^\sharp X_2
  \longleftrightarrow X_1\otimes X_2.
\]
In the symmetric monoidal category of modules over a commutative ring,
finite projective modules are dualizable and extension of scalars is a strong
symmetric monoidal functor: it sends duals, coevaluation, and evaluation
to their scalar extensions; see~\cite[Example~6.3]{ponto-shulman-traces}.
Consequently, it preserves categorical traces.  For a finite projective
$R_0$-module $P$ and an $R_0$-linear endomorphism $L\colon P\to P$,
applying~\cite[Proposition~6.2]{ponto-shulman-traces} to the extension of
scalars along $R_0\to R_1$ gives
\[
  \Tr_{R_1\otimes_{R_0}P/R_1}(1\otimes L)
  =
  1\otimes\Tr_{P/R_0}(L).
\]
Taking $L=m_{X_2}$ and identifying $R_1\otimes_{R_0}R_0\cong R_1$ by
$X_1\otimes Z\mapsto X_1\pi_1^\sharp Z$, we obtain
\[
\begin{split}
  \theta_{1*}\theta_2^\sharp X_2 &= \Tr_{R_{12}/R_1}(m_{\theta_2^\sharp X_2})
  = \Tr_{R_{12}/R_1}(1 \otimes m_{X_2}) \\ &= 1\otimes\Tr_{R_2/R_0}(m_{X_2})
  = \pi_1^\sharp(\pi_{2*}X_2).
\end{split}
\]
Thus, $\theta_{1*}\theta_2^\sharp=\pi_1^\sharp\pi_{2*}$, which is precisely the
Beck--Chevalley, or Mackey base-change, condition.  It says that transferring
along $\pi_2$ and then restricting along $\pi_1$ is the same as first
restricting to the fiber product along $\theta_2$ and then transferring
along $\theta_1$.

In this sense the assignment $\Omega\longmapsto \Rr^\Omega$ has the structure
of a finite commutative Green functor~\cite{bouc-green-functors}, and in
particular a Mackey functor~\cite{thevenaz-webb-mackey}: pullback $\pi^\sharp$
is the contravariant ring-valued part, transfer $\pi_*$ is the covariant
module-valued part, Frobenius reciprocity is the projection formula, and
Beck--Chevalley is the base-change identity for fiber products.

%

\section*{Declaration of Generative AI and AI-assisted technologies in the
writing process}
During the preparation of this work, the author used OpenAI's ChatGPT/Codex
to assist with language polishing, LaTeX editing, and exposition review.
After using this tool, the author reviewed and edited the content as needed
and takes full responsibility for the content of the publication.

\bibliographystyle{amsplain}
\bibliography{fprs.bib}


\end{document}